\documentclass[12pt,reqno]{amsart}
\usepackage{graphicx}
\usepackage{fullpage, amsmath, amssymb, amsthm}
\usepackage{hyperref}
\usepackage{mathtools}
\usepackage{colonequals}

\newtheorem{theorem}{Theorem}[section]
\newtheorem{proposition}[theorem]{Proposition}
\newtheorem{lemma}[theorem]{Lemma}
\newtheorem{corollary}[theorem]{Corollary}
\theoremstyle{definition}

\newtheorem{question}[theorem]{Question}
\theoremstyle{remark}
\newtheorem{remark}[theorem]{Remark}

\newcommand{\bF}{\mathbb{F}}
\newcommand{\Fp}{\mathbb{F}_p}
\newcommand{\Fq}{\mathbb{F}_q}
\newcommand{\bP}{\mathbb{P}}

\title{Covering plane curves over finite fields with lines}

\author{Shamil Asgarli}
\address{Department of Mathematics \& Computer Science \\ Santa Clara University \\ Santa Clara, CA 95053 \\ United States} \email{sasgarli@scu.edu}

\author{Sarthak Dassarma}
\address{Department of Mathematics \\ Columbia University \\ 116th and Broadway \\ New York, NY 10027 \\ United States}
\email{s.dassarma@columbia.edu}

\author{Dragos Ghioca}
\address{Department of Mathematics \\ University of British Columbia \\ Vancouver, BC V6T 1Z2\\ Canada}
\email{dghioca@math.ubc.ca}

\author{Simon MacLean}
\address{Department of Mathematics \\ Duke University \\ 120 Science Dr \\ Durham, NC 27710 \\ United States} \email{simon.maclean@duke.edu}

\author{Evan Stokdyk}
\address{Department of Mathematics \& Computer Science \\ Santa Clara University \\ Santa Clara, CA 95053 \\ United States} \email{estokdyk@scu.edu}



\begin{document}

\begin{abstract}
We prove that, for every fixed $d\geq 2$ and $\varepsilon>0$, the $\mathbb{F}_q$-points of a geometrically irreducible reflexive plane curve of degree $d$ over $\mathbb{F}_q$ can be covered by at most $\left(\frac{1}{d}+\varepsilon\right)q$ many $\mathbb{F}_q$-lines, provided that $q$ is sufficiently large. The same conclusion holds for $\mathbb{F}_q$-irreducible curves that are not geometrically irreducible, since they have only $O(d^2)$ $\mathbb{F}_q$-points. In particular, for sufficiently large $q$, the bound applies to every $\mathbb{F}_q$-irreducible plane curve of degree $d$ when $\mathrm{char}(\mathbb{F}_q)>d$. Our bound is asymptotically sharp. The proof combines Entin's Chebotarev density theorem for varieties over finite fields with the Pippenger--Spencer theorem on covers of nearly regular hypergraphs.
\end{abstract}

\maketitle

\section{Introduction}

Let $C \subset \bP^2$ be an $\Fq$-irreducible projective plane curve of degree $d\geq 2$ defined over $\Fq$. We are interested in the following natural question.

\begin{question}
What is the minimum number of $\Fq$-lines in $\bP^2$ whose union contains all points of $C(\Fq)$?
\end{question}

We write $\operatorname{cov}(C)$ for this minimum. If $C$ is not geometrically irreducible, then the problem is easy. Indeed, over $\overline{\Fq}$, the curve $C$ is the union of $r\geq 2$ Galois-conjugate irreducible components $C_1,\ldots,C_r$, each of degree $d/r$. Every point of $C(\Fq)$ lies on each of these components. In particular, $C(\Fq)\subset C_1\cap C_2$, and B\'ezout's theorem gives
\[
|C(\Fq)|\leq \left(\frac{d}{r}\right)^2\leq \frac{d^2}{4}.
\]
Consequently, $C(\Fq)$ can be covered by $O(d^2)$ $\Fq$-lines. We may therefore focus on the geometrically irreducible case.

By pairing up the points, we can always cover $C(\Fq)$ using $q/2+O_d(\sqrt{q})$ lines. However, as we will explain below, we can do much better.

B\'ezout's theorem shows that each line contains at most $d$ distinct $\Fq$-points of $C$. By the Aubry--Perret bound~\cite{AP96} for geometrically irreducible plane curves, we have $|C(\Fq)|\geq q+1-(d-1)(d-2)\sqrt q$. It follows that we need at least
\begin{equation}\label{eq:lower-bound}
\frac{q+1}{d} - \frac{(d-1)(d-2)\sqrt{q}}{d}
\end{equation}
$\Fq$-lines to cover $C(\Fq)$.

We write $C_{\mathrm{sm}}$ for the smooth locus of $C$. Given a geometrically irreducible plane curve $C$, the \emph{dual curve} $C^\vee\subset(\bP^2)^\vee$ parametrizes the tangent lines to $C$. We say that $C$ is \emph{reflexive} if its Gauss map
\[
C_{\mathrm{sm}}\longrightarrow C^\vee,
\qquad
P\longmapsto T_PC,
\]
is separable. 

Our main result is the following theorem.

\begin{theorem}\label{thm:main}
Fix $d\geq 2$ and $\varepsilon>0$. There exists a constant $M_{d,\varepsilon}$ such that the following holds. Let $C\subset\bP^2$ be an $\Fq$-irreducible plane curve of degree $d$ defined over $\Fq$. If $C$ is geometrically irreducible, assume in addition that $C$ is reflexive. If $q>M_{d,\varepsilon}$, then $C(\Fq)$ can be covered by at most $\left(\frac{1}{d}+\varepsilon\right)q$ many $\Fq$-lines.
\end{theorem}

\begin{corollary}\label{cor:large-characteristic}
The conclusion of Theorem~\ref{thm:main} holds for all $\Fq$-irreducible curves $C$ whenever $\operatorname{char}(\Fq)>d$.
\end{corollary}

\begin{proof}
If $C$ is not geometrically irreducible, then Theorem~\ref{thm:main} applies without any further assumption. Suppose, therefore, that $C$ is geometrically irreducible. A nonreflexive plane curve that is not a line has degree at least $\operatorname{char}(\Fq)$; see~\cite[p.~5]{Hef89}. Since $C$ is not a line and has degree $d<\operatorname{char}(\Fq)$, it must be reflexive. The conclusion now follows from Theorem~\ref{thm:main}.
\end{proof}

In view of the lower bound~\eqref{eq:lower-bound}, Theorem~\ref{thm:main} is asymptotically sharp.

\begin{remark}\label{rem:Fq-irreducibility-necessary}
The assumption that $C$ is $\Fq$-irreducible cannot be dropped. To see this, fix an integer $\ell\geq 2$ and, for $q\geq\ell$, choose distinct elements $c_1,\ldots,c_\ell\in\Fq$. Let $C$ be the union of the $\ell$ $\Fq$-conics obtained as the projective closures of
\[
y=x^2+c_i,\qquad 1\leq i\leq\ell.
\]
The sets of affine $\Fq$-points of these conics are pairwise disjoint, while all of the conics contain the point $[0:1:0]$ at infinity. Consequently,
\[
|C(\Fq)|=\ell q+1.
\]
On the other hand, any line meets each of the $\ell$ conics in at most two points and therefore contains at most $2\ell$ points of $C(\Fq)$. It follows that
\[
\operatorname{cov}(C)
\geq
\left\lceil\frac{\ell q+1}{2\ell}\right\rceil
>
\frac{q}{2}.
\]
Since $\deg(C)=2\ell$, the bound of Theorem~\ref{thm:main} fails for $C$ whenever $0<\varepsilon<\frac12-\frac{1}{2\ell}$.
\end{remark}

The proof of Theorem~\ref{thm:main} in the geometrically irreducible case has two main steps. First, we show that, outside a bounded exceptional set, the number of $\Fq$-lines through a smooth point $P\in C(\Fq)$ that meet $C$ in $d$ distinct smooth $\Fq$-points is
\[
\frac{q}{(d-1)!}+O_d(\sqrt q).
\]
This estimate is obtained by studying the monodromy of projection from $P$ and applying Entin's Chebotarev density theorem. We then form a $d$-uniform hypergraph whose vertices are the smooth $\Fq$-points of $C$ outside this exceptional set and whose edges are the sets of $d$ collinear vertices. The preceding estimate shows that this hypergraph is nearly regular, and the Pippenger--Spencer theorem produces the desired cover.

The intersection of lines and plane curves over finite fields is a recurring theme in arithmetic geometry and incidence combinatorics. Much of the recent literature concerns completely split line sections, which intersect the curve in the maximum possible number of distinct rational points. The existence of these lines also plays a key role in the proof of our main theorem. In related work, Diem~\cite{Die12} counts line-section divisors that split completely into $\Fq$-points on reflexive curves, and his \cite[Proposition~10]{Die12} closely parallels our Proposition~\ref{prop:degree}. From a different perspective, Makhul, Schicho, and Gallet~\cite{MSG20} use a function-field Chebotarev theorem to analyze the distribution of rational points within these line sections. While these earlier studies focus on counting or local distributions, we solve a global covering problem.

We conclude the introduction with a second motivation for the problem. When $d=3$ and $C$ is smooth with an $\Fq$-rational flex $O$, choosing $O$ as the identity makes $C$ an elliptic curve with the usual chord--tangent group law. In this case we can draw on existing results about abelian groups, in particular on the literature concerning the ``zero-sum property.''

For such an elliptic curve $E/\Fq$, we have
$E(\Fq)\cong\mathbb{Z}/a\mathbb{Z}\times\mathbb{Z}/b\mathbb{Z}$ with
$a\mid b$~\cite[Theorem~4.1, p.~97]{Was08}. A partition of $E(\Fq)$ into zero-sum triples yields a cover by the corresponding lines, since a line meets $E$ in three points summing to $O$ under the chord--tangent law. Writing $n=|E(\Fq)|$, if $n$ is odd, Tannenbaum's theorem~\cite{Tan81} on Skolem partitions gives
\[
\operatorname{cov}(E)=\frac{n}{3}+O(1).
\]
However, for $d>3$ the set $C(\Fq)$ has no analogous group structure, so we instead turn to tools from arithmetic geometry and the probabilistic method.

\section{Monodromy of projections from points}\label{sec:projection-monodromy}

In the geometrically irreducible case, the first step in the proof of Theorem~\ref{thm:main} is to estimate the number of $\Fq$-lines through a smooth point $P\in C(\Fq)$ that meet $C$ in $d$ distinct smooth $\Fq$-points. We do so by studying the projection from $P$.

This projection gives a finite cover of the pencil of lines through $P$, whose fiber over a line consists of the remaining $d-1$ points of intersection with $C$. All of these points are defined over $\Fq$ exactly when the Frobenius map acts trivially on the fiber. Entin's Chebotarev density theorem~\cite{Ent21} counts the fibers with this property once the monodromy group of the cover is known. We begin by recalling the necessary terminology.

Let $\phi\colon X\to Y$ be a finite \'etale morphism of degree $m$ between quasiprojective varieties defined over $\Fq$, where $Y$ is normal and geometrically irreducible. Fix a geometric point of $Y$; since $\phi$ is finite \'etale, the fiber over this point consists of $m$ distinct points. The \'etale fundamental group of $Y$ permutes these points, and choosing an ordering of the fiber turns this action into a homomorphism to the symmetric group $S_m$. A different ordering conjugates the image, so the resulting subgroup of $S_m$ is well defined up to conjugacy. We call this subgroup the \emph{arithmetic monodromy group} of $\phi$ and denote it by $\operatorname{Mon}(X/Y)$.

After extending the base field from $\Fq$ to $\overline{\Fq}$, the same construction gives the \emph{geometric monodromy group}
\[
\operatorname{Mon}^{\mathrm{g}}(X/Y)
=
\operatorname{Mon}\bigl(X\times_{\Fq}\overline{\Fq}\,/\,Y\times_{\Fq}\overline{\Fq}\bigr).
\]
The geometric monodromy group is a normal subgroup of the arithmetic monodromy group. More precisely, there is an integer $\nu\geq 1$ and an exact sequence
\begin{equation}\label{eq:monodromy-exact-sequence}
1\longrightarrow \operatorname{Mon}^{\mathrm{g}}(X/Y)
\longrightarrow \operatorname{Mon}(X/Y)
\longrightarrow \operatorname{Gal}(\bF_{q^\nu}/\Fq)
\longrightarrow 1,
\end{equation}
where $\bF_{q^\nu}$ is the smallest finite extension of $\Fq$ over which the arithmetic monodromy group of the base-changed cover coincides with the geometric monodromy group.

For a point $y\in Y(\Fq)$, the $q$-power Frobenius map permutes the $m$ geometric points of the fiber $\phi^{-1}(y)$. We call the conjugacy class of the resulting permutation in $\operatorname{Mon}(X/Y)$ the \emph{Frobenius class} of $\phi$ at $y$ and denote it by $\operatorname{Fr}_{\phi}(y)$. Its cycle lengths are the degrees of the closed points in the fiber. In particular, $\operatorname{Fr}_{\phi}(y)$ is the identity class precisely when the fiber consists of $m$ distinct $\Fq$-points. The group $\operatorname{Gal}(\bF_{q^\nu}/\Fq)$ is cyclic, generated by the automorphism $x\mapsto x^q$, and the final map in~\eqref{eq:monodromy-exact-sequence} sends every Frobenius class $\operatorname{Fr}_{\phi}(y)$ to this generator. 

The same definitions and hypotheses apply over any field $F$, with arithmetic monodromy over $F$ and geometric monodromy over an algebraic closure $\overline F$. For a finite separable map from a smooth projective integral curve to $\bP^1_F$, monodromy is the Galois group of the Galois closure of the extension of function fields. It stays the same after restricting to a nonempty open subset over which the map is \'etale.

The implied constant in Entin's error term depends only on the complexity of the cover, which we now define. A locally closed set $Z\subset\bP^N$ has \emph{complexity at most $M$} if $N\leq M$ and $Z$ can be described by equations $F_1=\cdots=F_s=0$ together with a single nonvanishing condition $G\neq 0$, where $F_1,\ldots,F_s$ and $G$ are homogeneous polynomials of degree at most $M$. The complexity of a morphism $\phi$ is the complexity of its graph, embedded in projective space via the Segre embedding. We write $\operatorname{comp}(\phi)$ for the complexity of $\phi$, and $O_{\operatorname{comp}(\phi)}(\cdot)$ denotes a bound whose implied constant depends only on $\operatorname{comp}(\phi)$. In particular, a bound on the complexity gives an estimate uniform in $q$.

We can now state the form of Entin's Chebotarev density theorem \cite[Theorem~3]{Ent21} that we need. It counts the points $y\in Y(\Fq)$ whose Frobenius class is a prescribed conjugacy class $\mathcal C$, under the necessary condition that $\mathcal C$ maps to the Frobenius generator of $\operatorname{Gal}(\bF_{q^\nu}/\Fq)$.

\begin{theorem}[Entin's Chebotarev density theorem]\label{thm:entin-chebotarev}
Let $\phi\colon X\to Y$ be a finite \'etale morphism of quasiprojective varieties defined over $\Fq$, where $Y$ is normal and geometrically irreducible. Let $\mathcal{C}$ be a conjugacy class in $\operatorname{Mon}(X/Y)$ whose image under the final map in~\eqref{eq:monodromy-exact-sequence} is the automorphism $x\mapsto x^q$. Then
\[
\bigl|\{y\in Y(\Fq):\operatorname{Fr}_{\phi}(y)=\mathcal{C}\}\bigr|
=
\frac{\nu|\mathcal{C}|}{|\operatorname{Mon}(X/Y)|}q^{\dim Y}
\left(1+O_{\operatorname{comp}(\phi)}(q^{-1/2})\right).
\]
Here $\nu\geq 1$ is the integer appearing in~\eqref{eq:monodromy-exact-sequence}.

When $Y$ is a curve, so that $\dim Y=1$, the error term is
$O_{\operatorname{comp}(\phi)}(\sqrt q)$.
\end{theorem}

We shall apply Theorem~\ref{thm:entin-chebotarev} to the following cover. For a smooth geometric point $P\in C$, let $\mathcal L_P\cong\bP^1$ be the pencil of lines through $P$. Let $U_P\subset\mathcal L_P$ be the open set of those lines $L$ for which $L\cap C$ consists of $d$ distinct smooth geometric points, and define
\[
X_P=\{(Q,L)\in C_{\mathrm{sm}}\times U_P:Q\in L,\ Q\neq P\}.
\]
Consider the projection
\[
\phi_P\colon X_P\longrightarrow U_P,
\qquad
(Q,L)\longmapsto L.
\]
For every $L\in U_P$, the fiber $\phi_P^{-1}(L)$ consists of the $d-1$ points of $(L\cap C)\setminus\{P\}$.

To count the completely split fibers, we will apply Theorem~\ref{thm:entin-chebotarev} with the identity conjugacy class. We first establish that $\phi_P$ is finite \'etale in Lemma~\ref{lem:incidence-cover}. Then Lemma~\ref{lem:generic-projection-monodromy} computes its geometric monodromy when $P$ is the generic point of $C$, and Lemma~\ref{lem:projection-monodromy} gives the geometric and arithmetic monodromy groups outside a set of points whose size is bounded in terms of $d$ alone.

\begin{lemma}\label{lem:incidence-cover}
Let $C\subset\bP^2$ be a geometrically irreducible reflexive curve of degree $d\geq2$ defined over $\Fq$, and put $k=\overline{\Fq}$. For every $P\in C_{\mathrm{sm}}(k)$, the set $U_P$ is nonempty and the map $\phi_P\colon X_P\to U_P$ is finite \'etale of degree $d-1$.
\end{lemma}

\begin{proof}
We work over $k$ and fix $P\in C_{\mathrm{sm}}(k)$. Let $\widetilde C$ be the normalization of $C$, and let $\widetilde P$ be the unique point above $P$. The pencil $\mathcal L_P$ cuts out a linear series on $\widetilde C$ with base point $\widetilde P$. Removing this base point leaves a base-point-free linear series and defines the residual map
\[
\pi_P\colon\widetilde C\longrightarrow\mathcal L_P
\]
of degree $d-1$. Concretely, if $Q\neq P$ is a smooth point of $C$ and $\widetilde Q$ is its lift to $\widetilde C$, then $\pi_P(\widetilde Q)$ is the line $PQ$, while $\pi_P(\widetilde P)$ is the tangent line $T_PC$.

We first check that $\pi_P$ is separable. A point $R\in\bP^2$ is called a \emph{strange point} for $C$ if every tangent line to $C$ at a smooth point passes through $R$. The curve $C$ is called \emph{strange} if it has a strange point. A reflexive curve of degree at least two has no strange point \cite[p.~4]{Hef89}, and projection from a point $R$ is separable on $C$ unless $R$ is a strange point; see the discussion before \cite[Lemma~2.2]{Ent21}. In particular, projection from $P$ is separable. The map $\pi_P$ induces the same extension of function fields, so it is separable as well.

Being separable, $\pi_P$ has only finitely many branch values. A line in $\mathcal L_P$ lies in $U_P$ exactly when it is not a branch value, is not the tangent line $T_PC$, and does not pass through a singular point of $C$. Only finitely many lines are excluded by these conditions, so $U_P$ is nonempty. Since $U_P$ contains no branch values, the restriction of $\pi_P$ over $U_P$ is finite \'etale of degree $d-1$, and the normalization map identifies it with $\phi_P\colon X_P\to U_P$.
\end{proof}

The maps $\phi_P$ and $\pi_P$ induce the same extension of
function fields, so they have the same monodromy group.
We now determine this group at the generic point of $C$.

\begin{lemma}\label{lem:generic-projection-monodromy}
Let $C$ and $k$ be as in Lemma~\ref{lem:incidence-cover}. After base change to $k$, let $\eta$ be the generic point of $C$ and set $K=k(C)$. Viewing $\eta$ as a smooth $K$-rational point of $C\times_k K$, define $\phi_\eta\colon X_\eta\to U_\eta$ by the same construction. Its monodromy group is $S_{d-1}$, both over $K$ and over an algebraic closure $\overline K$.
\end{lemma}

\begin{proof}
We keep the notation of the proof of Lemma~\ref{lem:incidence-cover}. The same proof over $K$, with $\eta$ in place of $P$, gives a separable residual map $\pi_\eta$ and a nonempty open set $U_\eta$ over which its restriction $\phi_\eta$ is finite \'etale of degree $d-1$. If $d=2$, then $\pi_\eta$ has degree one and trivial monodromy over both $K$ and $\overline K$. We may assume from now on that $d\geq3$. We first compute the monodromy group over $K$, and then show that it is the same over $\overline K$.

Let $\Omega\subset(\bP^2)^\vee$ be the open set of lines that meet $C$ in $d$ distinct smooth points, and let $V$ be the surface of pairs $(P,L)$ with $L\in\Omega$ and $P\in L\cap C$. The projection $V\to\Omega$ is finite \'etale of degree $d$, and its monodromy group is $S_d$ by results of Rathmann~\cite{Rat87} and of Ballico and Hefez~\cite{BH86}; see \cite[Theorem~2]{Ent21}. 

To describe the remaining intersection points after $P$ has
been marked, consider the cover $V'\to V$ of triples $(P,Q,L)$
with $(P,L)\in V$ and $Q\in(L\cap C)\setminus\{P\}$.
This cover has degree $d-1$. Since $V$ marks one of the $d$
intersection points, the monodromy group of $V'\to V$ is the
stabilizer $S_{d-1}$ of that point in $S_d$, acting on the
remaining $d-1$ points.

To relate this cover to $\phi_\eta$, we take generic fibers
of the projection $V\to C_{\mathrm{sm}}$, $(P,L)\mapsto P$.
Its generic fiber is $U_\eta$, a nonempty open subset of the pencil
$\mathcal L_\eta\cong\bP^1_K$. Over this fiber, the cover
$V'\to V$ parametrizes pairs $(Q,L)$ with $L\in U_\eta$ and
$Q\in(L\cap C)\setminus\{\eta\}$, and sends $(Q,L)$ to $L$.
This is precisely $\phi_\eta\colon X_\eta\to U_\eta$.
Passing to the generic fibers does not change the extension
of function fields, so $k(V')/k(V)$ is the extension induced
by $\phi_\eta$, and hence also by $\pi_\eta$.
It follows that the monodromy group of $\pi_\eta$ over $K$ is $S_{d-1}$.

Since $K$ is not algebraically closed, this group may become smaller after base change to $\overline K$. Write $\bar\eta$ for the resulting $\overline K$-point and $\pi_{\bar\eta}$ for the base-changed map, and let $H$ be its monodromy group. The geometric monodromy group is normal in the monodromy group over $K$, so $H\trianglelefteq S_{d-1}$. We will now show that $H=S_{d-1}$. The transpositions form a single conjugacy class of $S_{d-1}$ and generate it, so it suffices to find one transposition in $H$.

Since $C$ is reflexive, the characteristic is not two, and a general tangent line has contact order two at one smooth point and meets $C$ transversely at $d-2$ further smooth points \cite[Section~2]{Ent21}. Call such a line a \emph{simple tangent}. To obtain the desired transposition, we seek a simple tangent
that meets $C$ transversely at $\bar\eta$.

Let $I$ be the variety of pairs $(R,L)$ where $L$ is a simple tangent and $R$ is a point where $L$ meets $C$ transversely. Let $\Theta\subset C$ be the image of $I\to C$, $(R,L)\mapsto R$. This image is constructible, and we claim that it is dense. If not, then $\Theta$ is finite. Since $d\geq3$, every simple tangent passes through a point of $\Theta$. For each $R\in\Theta$, the set $\{L\in C^\vee:R\in L\}$ is closed in $C^\vee$. These finitely many closed subsets together contain all simple tangents, which form a dense subset of the irreducible curve $C^\vee$, so one of them is all of $C^\vee$. The corresponding point $R$ would then be a strange point of $C$. This is impossible for a reflexive curve, as recalled in the proof of Lemma~\ref{lem:incidence-cover}.

The projection $I\to C$ is therefore dominant, so its generic fiber is nonempty. Since this fiber is of finite type over $K$, it has a $\overline K$-point, giving a simple tangent $L_0$ that meets $C$ transversely at $\bar\eta$. The fiber of $\pi_{\bar\eta}$ over $L_0$ consists of the point of tangency, with ramification index two, and $d-3$ unramified points. This ramification is tame because $\operatorname{char}(k)\neq2$, and it yields a transposition in $H$, which forces $H=S_{d-1}$. The same holds for $\phi_\eta$, since restricting $\pi_\eta$ to $U_\eta$ does not change its monodromy group.
\end{proof}

The next lemma carries this conclusion from the generic point to every smooth point of $C$ outside a finite set, whose size we bound in terms of the degree.

\begin{lemma}\label{lem:projection-monodromy}
Let $C\subset\bP^2$ be a geometrically irreducible reflexive curve of degree $d\geq2$ defined over $\Fq$. There exists a set $B_C\subset C_{\mathrm{sm}}(\overline{\Fq})$, with $|B_C|=O_d(1)$, such that the following holds for every
$P\in C_{\mathrm{sm}}(\overline{\Fq})\setminus B_C$.

The map $\phi_P\colon X_P\to U_P$ is finite \'etale of degree $d-1$, and its geometric monodromy group is $S_{d-1}$. If $P\in C(\Fq)$, then the arithmetic monodromy group of $\phi_P$ is also $S_{d-1}$. Moreover, the complexity of $\phi_P$ is $O_d(1)$, uniformly in $C$, $q$, and $P$.
\end{lemma}

\begin{proof}
Work over $k=\overline{\Fq}$ and put $K=k(C)$. We first show that the monodromy group of $\phi_P$ is $S_{d-1}$ for all but finitely many $P\in C_{\mathrm{sm}}(k)$. We express this condition as the irreducibility of an auxiliary cover, because geometric irreducibility passes from the generic fiber to the fibers over an open set.

For $P\in C_{\mathrm{sm}}(k)$, let $W_P\to U_P$ be the cover obtained by ordering the $d-1$ points in each fiber of the finite \'etale map $\phi_P$ from Lemma~\ref{lem:incidence-cover}. The monodromy group $G_P$ of $\phi_P$ acts freely on the $(d-1)!$ orderings of a fiber, and the irreducible components of $W_P$ correspond to the $(d-1)!/|G_P|$ orbits. In particular, $W_P$ is irreducible exactly when $G_P=S_{d-1}$.

Let $W$ be the variety of tuples $(P,Q_1,\ldots,Q_{d-1},L)$ in which $L$ meets $C$ in the $d$ distinct smooth points $P,Q_1,\ldots,Q_{d-1}$. The fiber of the projection
\[
W\longrightarrow C_{\mathrm{sm}},
\qquad (P,Q_1,\ldots,Q_{d-1},L)\longmapsto P
\]
over $P$ is $W_P$, and its generic fiber is the ordered cover $W_\eta$ attached to $\phi_\eta$. By Lemma~\ref{lem:generic-projection-monodromy} and the criterion above, applied over $\overline K$, the generic fiber $W_\eta$ is geometrically irreducible. This property extends to the fibers over a dense open subset of $C_{\mathrm{sm}}$ \cite[Tag~0559]{Stacks}. The complement of a dense open subset of a curve is finite, so the set
\[
B_C=\{P\in C_{\mathrm{sm}}(k):W_P\text{ is not irreducible}\}
\]
is finite. By the criterion above, the monodromy group of $\phi_P$ is $S_{d-1}$ for every $P\in C_{\mathrm{sm}}(k)\setminus B_C$. When $d=2$, the cover $W_P\to U_P$ has degree one for every $P$, so $B_C=\varnothing$.

It remains to bound $|B_C|$ in terms of $d$ alone, and for this we let the curve vary. Let $\mathcal P$ be the projective space over $\mathbb Z$ parametrizing plane curves of degree $d$, and let $\mathcal U\to\mathcal P$ be the universal family of pairs $(C,P)$ with $P$ smooth on $C$. The construction of $W$ gives a family of finite type over $\mathcal U$ whose fiber over $(C,P)$ is $W_P$. Let $\mathcal E\subset\mathcal U$ be the set of points over which this fiber is not geometrically irreducible. This set is constructible \cite[Tag~055B]{Stacks}, and for a curve $C$ as in the lemma its geometric fiber over $[C]\in\mathcal P$ is $B_C$.

Write $\mathcal E$ as a finite union of locally closed subschemes $\mathcal E_1,\ldots,\mathcal E_s$. For each $i$, the set of points at which $\mathcal E_i\to\mathcal P$ is quasi-finite is open in $\mathcal E_i$ \cite[Tag~01TI]{Stacks}. For a curve $C$ as in the lemma, the geometric fiber of each $\mathcal E_i$ over $[C]$ is a subset of the finite set $B_C$. Each point is isolated in that fiber and therefore lies in the quasi-finite locus. This open locus is quasi-compact because $\mathcal E_i$ is Noetherian, so there is an integer $N_i$ bounding the number of points in each of its geometric fibers over $\mathcal P$ \cite[Tag~03JA]{Stacks}. We obtain $|B_C|\leq N_1+\cdots+N_s$. Since the whole construction takes place over $\mathbb Z$ for fixed $d$, the integers $s$ and $N_i$ depend only on $d$, which gives $|B_C|=O_d(1)$.

Now suppose that $P\in C_{\mathrm{sm}}(\Fq)\setminus B_C$. Then $\phi_P$ is defined over $\Fq$, and its arithmetic monodromy group is a subgroup of $S_{d-1}$ that contains the geometric monodromy group. We have just seen that the latter is all of $S_{d-1}$, so the two groups coincide.

It remains to bound the complexity of $\phi_P$, and here $P$ may be any smooth point of $C$. Choose projective coordinates with $P=[0:1:0]$ and $T_PC=\{z=0\}$. If $F=0$ defines $C$, these coordinate choices force the coefficients of $y^d$ and $xy^{d-1}$ in $F$ to vanish, while the coefficient of $zy^{d-1}$ is nonzero. After scaling $F$, the polynomial $f(x,y)=F(x,y,1)$ is therefore monic of degree $d-1$ in $y$.

Let $D(x)=\operatorname{Disc}_y(f)(x)$ be the discriminant of $f$ with respect to $y$. The lines of the pencil other than $T_PC$ have equations $x=az$. Such a line lies in $U_P$ exactly when $f(a,y)$ has $d-1$ distinct roots, or equivalently when $D(a)\neq0$; at each root, $f_y\neq0$, so the corresponding point of $C$ is smooth. In these coordinates, $\phi_P$ is the projection $(x,y)\mapsto x$ from $\{f=0,\ D\neq0\}$ to $\{D\neq0\}$. The degrees of $f$ and $D$ are bounded in terms of $d$, so the graph of $\phi_P$ has complexity $O_d(1)$, uniformly in $C$, $q$, and $P$.
\end{proof}

We can now apply Entin's theorem to count the completely split fibers of $\phi_P$. These fibers correspond precisely to the $\Fq$-lines through $P$ that meet $C$ in $d$ distinct smooth $\Fq$-points.

\begin{proposition}\label{prop:degree}
Let $C$ and $B_C$ be as in Lemma~\ref{lem:projection-monodromy}, and let
$P\in C_{\mathrm{sm}}(\Fq)\setminus B_C$. The number of $\Fq$-lines
through $P$ that meet $C$ in $d$ distinct smooth $\Fq$-points is
\[
\frac{q}{(d-1)!}+O_d(\sqrt q).
\]
The implied constant is uniform in $C$, $q$, and $P$.
\end{proposition}

\begin{proof}
Apply Theorem~\ref{thm:entin-chebotarev} to the finite \'etale morphism
$\phi_P\colon X_P\to U_P$. By Lemma~\ref{lem:projection-monodromy}, both the geometric and arithmetic monodromy groups of $\phi_P$ are $S_{d-1}$. The exact sequence~\eqref{eq:monodromy-exact-sequence} then gives $\nu=1$, and the identity conjugacy class satisfies the hypothesis of Entin's theorem.

A line $L\in U_P(\Fq)$ has Frobenius class equal to the identity precisely when every point of the fiber $\phi_P^{-1}(L)$ is defined over $\Fq$. Since this fiber is $(L\cap C)\setminus\{P\}$, that happens exactly when $L$ meets $C$ in $d$ distinct smooth $\Fq$-points.

The identity conjugacy class has size one, while $|S_{d-1}|=(d-1)!$. Since $U_P$ is a curve and the complexity of $\phi_P$ is $O_d(1)$, Theorem~\ref{thm:entin-chebotarev} gives
\[
\frac{q}{(d-1)!}+O_d(\sqrt q)
\]
such $\Fq$-lines. The complexity bound is uniform in $C$, $q$, and $P$, so the same is true of the implied constant.
\end{proof}

A closely related estimate is stated in
\cite[Proposition~10]{Die12}, under the hypothesis that only
ordinary tangents pass through $P$. For singular plane models, projection may ramify at a branch even when the corresponding line is not tangent to that branch. We have given an independent argument in the proofs of
Lemmas~\ref{lem:generic-projection-monodromy} and~\ref{lem:projection-monodromy} that accounts for this
possibility and yields the uniform estimate used in
Section~\ref{sec:main-thm}.

\section{Proof of Theorem~\ref{thm:main}}\label{sec:main-thm}

We first dispose of the case in which $C$ is not geometrically irreducible. As observed in the introduction, in this case
\[
|C(\Fq)|\leq \frac{d^2}{4}.
\]
These points can be covered by at most $d^2/4$ $\Fq$-lines, and $d^2/4 \le \left(\frac{1}{d}+\varepsilon\right)q$ after enlarging $M_{d,\varepsilon}$ if necessary. We may therefore assume that $C$ is geometrically irreducible. By the hypothesis of Theorem~\ref{thm:main}, the curve $C$ is then reflexive.

We now complete the proof by combining the estimate from Proposition~\ref{prop:degree} with the Pippenger--Spencer theorem~\cite{PS89} on covers of nearly regular hypergraphs.

A \emph{cover} of a hypergraph is a set of edges whose union contains all vertices. Here $d(v)$ denotes the degree of a vertex $v$, and $d(v,w)$ the number of edges containing both $v$ and $w$. We use the following form of the Pippenger--Spencer theorem, taken from Alon and Spencer~\cite[Theorem~4.7.1]{AS16}.

\begin{theorem}\label{thm:pippenger-spencer}
For every integer $r \geq 2$ and all reals $k \geq 1$ and $a > 0$, there are $\gamma=\gamma(r,k,a)>0$ and
$d_0=d_0(r,k,a)$ such that for every $n\in\mathbb{N}$ and $D \geq d_0$ the following holds.

Let $H=(V,E)$ be an $r$-uniform hypergraph on $n$ vertices in which every vertex has positive degree. Suppose that
\begin{enumerate}
    \item all but at most $\gamma n$ vertices $v\in V$ satisfy $(1-\gamma)D\leq d(v)\leq (1+\gamma)D$;
    \item every vertex $v\in V$ satisfies $d(v)<kD$;
    \item every pair of distinct vertices $v,w\in V$ satisfies $d(v,w)<\gamma D$.
\end{enumerate}
Then $H$ has a cover consisting of at most $(1+a)n/r$ edges.
\end{theorem}

Let $S\subset C(\Fq)$ be the set of singular $\Fq$-points, and let $B_C$ be the exceptional set from Lemma~\ref{lem:projection-monodromy}. Both $|S|$ and $|B_C|$ are bounded in terms of $d$, so the set
\[
T=S\cup\bigl(B_C\cap C(\Fq)\bigr)
\]
has size $O_d(1)$. We first cover the points of $T$: each lies on some $\Fq$-line, so at most $|T|=O_d(1)$ lines suffice. It remains to cover the points of $C(\Fq)\setminus T$.

To apply Theorem~\ref{thm:pippenger-spencer}, we describe the hypergraph $H=(V,E)$. The vertex set is $V=C(\Fq)\setminus T$, and the edges are the sets of $d$ distinct collinear points of $V$. Equivalently, each edge arises from an $\Fq$-line meeting $C$ in $d$ distinct smooth $\Fq$-points, none of which lies in $T$. Thus, $H$ is a $d$-uniform hypergraph.

We compute the degree of a vertex $v\in V$, which corresponds to a point $P\in C(\Fq)\setminus T$. By Proposition~\ref{prop:degree}, the number of $\Fq$-lines through $P$ that meet $C$ in $d$ distinct smooth $\Fq$-points is
\[
\frac{q}{(d-1)!}+O_d(\sqrt q).
\]
Among these lines, at most $|T|=O_d(1)$ of them contain a point of $T$, since each point of $T$ determines a unique line through $P$. After these lines are discarded, the remaining lines meet $C$ in $d$
distinct smooth $\Fq$-points of $V$, and each resulting set of points
is an edge of $H$ containing $v$. We obtain
\begin{equation}\label{eq:degree-of-a-vertex}
d(v)=\frac{q}{(d-1)!}+O_d(\sqrt q)
\end{equation}
for every $v\in V$. In particular, every vertex has positive degree when $q$ is sufficiently large.

We apply Theorem~\ref{thm:pippenger-spencer} with the following parameters. Fix an arbitrary $a>0$, and set $r=d$ and $k=2(d-1)!$. With these choices, $d_0(r,k,a)$ is a constant, so $D=\frac{q}{(d-1)!}$ exceeds $d_0$ once $q$ is sufficiently large; we take this value of $D$. Finally, the Aubry--Perret bound~\cite{AP96} gives $|C(\Fq)|=q+O_d(\sqrt q)$, so $n=|V|=|C(\Fq)|-|T|=q+O_d(\sqrt q)$.

We now verify the three conditions. Here $\gamma=\gamma(r,k,a)$ and $d_0=d_0(r,k,a)$ denote the constants supplied by Theorem~\ref{thm:pippenger-spencer}.
\begin{enumerate}
    \item Using~\eqref{eq:degree-of-a-vertex}, we obtain
    \[
    \frac{q}{(d-1)!}\left(1-O_d\left(\frac{1}{\sqrt{q}}\right)\right)
    \leq d(v)\leq
    \frac{q}{(d-1)!}\left(1+O_d\left(\frac{1}{\sqrt{q}}\right)\right).
    \]
    Since $\gamma>0$ is a fixed constant and $D=\frac{q}{(d-1)!}$, this gives $(1-\gamma)D\leq d(v)\leq(1+\gamma)D$ for sufficiently large $q$. This holds for every vertex of $H$, so the exceptional set allowed in condition~(1) is in fact empty.
    \item Every point of $\bP^2(\Fq)$ lies on exactly $q+1$ distinct $\Fq$-lines, so $d(v)\leq q+1$ for every $v$. Since
    \[
    kD=2(d-1)!\cdot\frac{q}{(d-1)!}=2q,
    \]
    we have $d(v)\leq q+1<2q=kD$ for every $v\in V$.
    \item Two distinct $\Fq$-points lie on a unique $\Fq$-line, and any edge containing both $v$ and $w$ consists of $d$ collinear points on that line. Hence at most one edge contains both of them, so $d(v,w)\leq 1$. Since $\gamma D=\gamma\frac{q}{(d-1)!}\to\infty$, we have $d(v,w)<\gamma D$ for all sufficiently large $q$.
\end{enumerate}
By Theorem~\ref{thm:pippenger-spencer}, the hypergraph $H$ has a cover
consisting of at most
\[
(1+a)\frac{n}{d}
\]
edges. Since $n=|V|=q+O_d(\sqrt q)$, there is a constant $K_d$ such that $n\leq q+K_d\sqrt q$ for all sufficiently large $q$. The number of edges in the cover is at most
\[
(1+a)\frac{q+K_d\sqrt q}{d}
=
\frac{q}{d}+\frac{aq}{d}+\frac{(1+a)K_d}{d}\sqrt q.
\]
Let $R_d$ be a constant such that the points of $T$ can be covered using at most $R_d$ additional $\Fq$-lines. Recall that $a>0$ was arbitrary. Given $\varepsilon>0$, take $a$ small enough that $a/d<\varepsilon/3$, and then $q$ large enough that
\[
\frac{(1+a)K_d}{d}\sqrt q<\frac{\varepsilon q}{3}
\qquad\text{and}\qquad
R_d<\frac{\varepsilon q}{3}.
\]
Combining the hypergraph cover with the additional lines used to cover $T$, the total number of $\Fq$-lines is at most
\[
\frac{q}{d}+\frac{\varepsilon q}{3}+\frac{\varepsilon q}{3}+\frac{\varepsilon q}{3}
=
\left(\frac{1}{d}+\varepsilon\right)q.
\]
This proves Theorem~\ref{thm:main}.

\section{An obstruction in every positive characteristic}\label{sec:obstruction}

In this section, we show that the reflexivity hypothesis in
Theorem~\ref{thm:main} cannot be omitted in any positive characteristic. For every prime $p$, we exhibit a geometrically irreducible nonreflexive curve of degree $p^2$, defined over $\Fp$, whose covering number over $\bF_q$, for $q=p^{2n+1}$, is $q/p+1$. These curves fall outside Theorem~\ref{thm:main} because they are nonreflexive, and outside Corollary~\ref{cor:large-characteristic} because $p^2>p$.

The obstruction is visible in the proof itself. The application of Entin's theorem in Section~\ref{sec:projection-monodromy} depends on the identity conjugacy class satisfying the compatibility condition in Theorem~\ref{thm:entin-chebotarev}. This condition holds exactly when $\nu=1$, and in general one may have $\nu>1$. The following example shows what can go wrong in every positive characteristic.

\begin{proposition}\label{prop:nonreflexive-obstruction}
Let $p$ be a prime, let $q=p^{2n+1}$, where $n\geq 0$, and let
$C\subset\bP^2$ be the curve of degree $p^2$ given by
\[
yz^{p^2-1}=x^{p^2}.
\]
Then $C$ is geometrically irreducible and nonreflexive, and
$|C(\Fq)|=q+1$, but every $\Fq$-line contains at most $p$ points of
$C(\Fq)$. Moreover,
\[
\operatorname{cov}(C)
=
\left\lceil\frac{q+1}{p}\right\rceil
=
\frac{q}{p}+1.
\]
\end{proposition}

\begin{proof}
Regard $yz^{p^2-1}-x^{p^2}$ as a polynomial in $y$ with coefficients in
$\overline{\Fq}[x,z]$. It is linear in $y$, and its coefficients
$z^{p^2-1}$ and $x^{p^2}$ are coprime, so it is irreducible in
$\overline{\Fq}[x,y,z]$. This proves that $C$ is geometrically irreducible. It is the projective closure of the affine rational curve
\[
y=x^{p^2}.
\]
The curve is parametrized by
\[
[s:t]\longmapsto
[st^{p^2-1}:s^{p^2}:t^{p^2}].
\]
On the affine chart $t=1$, this gives the points
$[s:s^{p^2}:1]$, while $t=0$ gives the point
$P_\infty=[0:1:0]$. These exhaust $C(\Fq)$, so
\[
C(\Fq)
=
\{[s:s^{p^2}:1]:s\in\Fq\}\cup\{P_\infty\},
\]
and $|C(\Fq)|=q+1$.

The point $P_\infty$ is singular, whereas every affine point is smooth. At the point $[s:s^{p^2}:1]$, the tangent line is $y=s^{p^2}z$. All these tangent lines pass through $[1:0:0]$. We conclude that $C$ is a strange curve, and therefore nonreflexive.

Let $L$ be an $\Fq$-line with equation $ax+by+cz=0$. Suppose first that $b=0$, so $L$ has equation $ax+cz=0$ with $(a,c)\neq(0,0)$. Then $P_\infty$ lies on $L$, and substituting $[s:s^{p^2}:1]$ into the equation gives $as+c=0$. This linear equation has at most one solution $s\in\Fq$, so $L$ contains at most two points of $C(\Fq)$, hence at most $p$ since $p\ge 2$.

Now suppose that $b\neq0$. Then $L$ does not contain $P_\infty$, and its
intersections with the affine part of $C$ correspond to the roots in
$\Fq$ of
\[
bs^{p^2}+as+c=0.
\]
Set $A(T)=bT^{p^2}+aT$, an additive polynomial. If
$A(r)=-c$, then the solutions of $A(s)=-c$ are exactly the elements of
the coset $r+\ker(A)$.

It is enough to bound the size of $\ker(A)$. If $a=0$, then
$\ker(A)=\{0\}$. Suppose that $a\neq 0$. If $0\neq u\in\ker(A)$, then
\[
u^{p^2-1}=-\frac{a}{b}.
\]
If this equation has a solution $u_0\in\Fq^\times$, then all of its solutions are of the form $u_0w$, where $w^{p^2-1}=1$. Since
\[
\gcd(p^2-1,q-1) = \gcd(p^2-1,p^{2n+1}-1) = p^{\gcd(2,2n+1)}-1 = p-1,
\]
there are exactly $p-1$ solutions to the latter equation in
$\Fq^\times$. Therefore, $\ker(A)$ contains at most $p-1$ nonzero elements, yielding $|\ker(A)|\leq p$. Since every nonempty fiber of $A$ is a coset of $\ker(A)$, the equation $bs^{p^2}+as+c=0$ has at most $p$ roots in $\Fq$.

We have shown that every $\Fq$-line contains at most $p$ points of $C(\Fq)$. Since $|C(\Fq)|=q+1$, covering these points requires at least
\[
\left\lceil\frac{q+1}{p}\right\rceil = \frac{q}{p}+1
\]
lines. For the reverse inequality, consider the additive map
\[
B\colon\Fq\longrightarrow\Fq,
\qquad B(s)=s^{p^2}-s.
\]
Its kernel is
\[
\ker(B) = \bF_{p^2}\cap\Fq = \Fp,
\]
so $|\operatorname{im}(B)|=q/p$. For each $r\in\operatorname{im}(B)$, take the line
\[
L_r\colon y=x+rz.
\]
Every affine point $[s:s^{p^2}:1]$ lies on $L_{B(s)}$. These $q/p$ lines cover all affine points of $C(\Fq)$, and the line $z=0$ covers $P_\infty=[0:1:0]$. It follows that
\[
\operatorname{cov}(C)\leq\frac{q}{p}+1.
\]
Together with the lower bound established above, this proves
\[
\operatorname{cov}(C)=\frac{q}{p}+1,
\]
as claimed. \end{proof}

Since $C$ has degree $p^2$ while
\[
\operatorname{cov}(C)=\frac{q}{p}+1,
\]
these examples do not obey any bound of the form
$\left(\frac{1}{p^2}+\varepsilon\right)q$ when
$0<\varepsilon<\frac{1}{p}-\frac{1}{p^2}$. This confirms that the reflexivity hypothesis in Theorem~\ref{thm:main} is essential.

\begin{remark}\label{rem:monodromy-obstruction}
The preceding example also explains why the condition $\nu=1$ in the
application of Entin's theorem is not automatic. The smooth $\Fq$-points
of $C$ are
\[
P_a=[a:a^{p^2}:1],
\qquad
a\in\Fq.
\]
Every line through $P_a$, other than the line $x=az$, can be written as
\[
y-a^{p^2}z=m(x-az).
\]
Substituting $[s:s^{p^2}:1]$ into this equation gives
\[
(s-a)^{p^2}=m(s-a).
\]
Set $u=s-a$. The root $u=0$ corresponds to the fixed point $P_a$.
Removing this root and restricting to $m\neq0$ leaves the equation
\[
u^{p^2-1}=m
\]
for the remaining points. Under the identification of $U_{P_a}$ with $\mathbb G_m$, the residual cover $\phi_{P_a}$ is
\[
\mathbb G_m\longrightarrow\mathbb G_m, \qquad u\longmapsto u^{p^2-1}.
\]

Over $\overline{\Fq}$, this cover is Galois with geometric monodromy
group $C_{p^2-1}$. All $(p^2-1)$-th roots of unity lie in
$\bF_{p^2}$, and a primitive one generates $\bF_{p^2}$ over $\Fp$.
Moreover, $\bF_{p^2}\cap\Fq=\Fp$. Consequently, the smallest extension of $\Fq$ containing all these
roots of unity is $\bF_{q^2}$. The Galois closure of
$\Fq(u)/\Fq(m)$, where $m=u^{p^2-1}$, is $\bF_{q^2}(u)$. The arithmetic monodromy group is $C_{p^2-1}\rtimes C_2$, where the nontrivial element of $C_2$ acts by
\[
\zeta\longmapsto\zeta^q=\zeta^p.
\]
Here we used $q=p^{2n+1}\equiv p\pmod{p^2-1}$.
For $p=2$, this specializes to
$C_3\rtimes C_2\cong S_3$. In general, the exact
sequence~\eqref{eq:monodromy-exact-sequence} becomes
\[
1\longrightarrow C_{p^2-1}
\longrightarrow C_{p^2-1}\rtimes C_2
\longrightarrow \operatorname{Gal}(\bF_{q^2}/\Fq)
\longrightarrow 1.
\]
In particular, $\nu=2$, not $1$.

For a parameter $m\in\Fq^\times$, the Frobenius class of the fiber over $m$ maps to the nontrivial element of
$\operatorname{Gal}(\bF_{q^2}/\Fq)$, so it lies in the nontrivial coset of $C_{p^2-1}$ in the arithmetic monodromy group. The identity, however, lies in $C_{p^2-1}$ itself, so it cannot occur as the Frobenius class of any $\Fq$-fiber. Equivalently, the residual cover has no completely split fiber over $\Fq$. This is exactly the compatibility condition in Theorem~\ref{thm:entin-chebotarev}: the identity maps to the trivial element of $\operatorname{Gal}(\bF_{q^\nu}/\Fq)$, which is the Frobenius generator only when $\nu=1$. By
Lemma~\ref{lem:projection-monodromy}, for the covers $\phi_P$ used in the proof of Theorem~\ref{thm:main}, the geometric and arithmetic monodromy groups are both $S_{d-1}$. As a result, their quotient is trivial, so $\nu=1$ and the identity class is allowed.
\end{remark}

\section*{Acknowledgments}
The third author was partially supported by an NSERC Discovery Grant. The fourth author was partially supported by the REAL program at Santa Clara University (SCU). The second and fifth authors were partially supported by the REAL program and the Rick Scott Memorial Scholarship at SCU.

\end{document}